\documentclass[a4paper,10pt]{article}
\usepackage[utf8]{inputenc}
\usepackage{amsmath}
\usepackage{amsfonts}
\usepackage{amssymb}
\usepackage{amsthm} 
\usepackage{csquotes}
\usepackage[shortlabels]{enumitem}
\usepackage[english]{babel}
\usepackage{cite}
\usepackage{keyval}

\usepackage{xcolor}

\title{A dimensional restriction for a class of contact manifolds}
\author{Eugenia Loiudice}
\date{}

\newtheorem*{theorem*}{Theorem}
\newtheorem{theorem}{Theorem}

\newtheorem{rem}{Remark}
\newtheorem{lemma}{Lemma}
\newtheorem{prop}{Proposition}
\newtheorem{cor}{Corollary}

\DeclareMathOperator{\rank}{rank}
\DeclareMathOperator{\spn}{span}
\newcommand{\di}{\mbox{d}}

\begin{document}

\maketitle

\begin{abstract}

In this work we consider a class of contact manifolds $(M,\eta)$ with an associated almost contact metric structure $(\phi, \xi, \eta,g)$. This class contains, for example, nearly cosymplectic manifolds and the manifolds in the class $C_9\oplus C_{10}$ defined by Chinea and Gonzalez.

All manifolds in the class considered turn out to have dimension $4n+1$. Under the assumption that the sectional curvature of the horizontal $2$-planes is constant at one point, we obtain that these manifolds must have dimension $5$.

\end{abstract}

\vspace{3 mm}


\section{Introduction}

A \emph{contact manifold} is a $\mathcal{C}^\infty$ odd-dimensional manifold $M^{2n+1}$ together with a $1-$form $\eta$, usually called a \emph{contact form} on $M$, such that $\eta\wedge (d\eta)^n \neq 0$ everywhere on $M;$ the \emph{contact distribution} $D$ is the vector subbundle of $TM$ defined by  
$$
D:=\ker \eta.
$$ 
We shall denote by $D_p$ the fiber of $D$ at a point $p$; moreover if $X\in \mathfrak{X}(M)$ is a vector field, we shall write $X\in D$ to indicate that $X$ is a section of $D$.

It is known that $\di\eta|_{D_p\times D_p}$ is non degenerate and 
$$
T_p M=D_p\oplus \ker \di\eta _p
$$
for each $p\in M$.

In \cite{chern} Chern showed that the existence of a contact form $\eta$ on a manifold $M^{2n+1}$ implies that the structural group of the tangent bundle $TM$ can be reduced to the unitary group $U(n)\times 1$. Such a reduction of the structural group of the tangent bundle of a manifold $M^{2n+1}$ is called an \emph{almost contact structure}. In term of structure tensors we say that an \emph{almost contact structure} on a manifold $M^{2n+1}$ is a triple $(\phi,\xi,\eta)$ consisting of a tensor field $\phi$ of type $(1,1),$ a vector field $\xi$ and a $1-$form $\eta$ satisfying 
$$
\phi ^2=-I+\eta \otimes \xi, \quad \eta(\xi)=1,
$$
see \cite{blair2010riemannian} p. 43. It then follows directly from the definition of almost contact structure that $\phi \xi=0, \eta \circ \phi =0$, and that the endomorphism $\phi$ has rank $2n.$ If, in addition, $M$ is endowed with a Riemannian metric $g$ such that 
$$
g(\phi X,\phi Y)=g(X,Y)-\eta(X)\eta(Y),
$$
then $(\phi,\xi,\eta,g)$ is said to be an \emph{almost contact metric structure} on $M.$
Thus, setting $Y=\xi,$ we have immediately that 
$$
\eta (X)=g(X,\xi).
$$

Every contact manifold $(M^{2n+1},\eta)$ admits an almost contact metric structure $(\phi,\xi,\eta,g)$ such that 
$$
\di\eta(X,Y)=g(X, \phi Y).
$$
In this case $g$ is an \emph{associated metric} and we speak of a \emph{contact metric structure}; the vector field $\xi$ is the Reeb vector field of $M^{2n+1}$ \cite{blair2010riemannian}.  
Of course, it is possible to have a contact manifold $(M^{2n+1},\eta)$ with Reeb vector field $\xi$ and an almost contact metric structure $(\phi,\xi,\eta,g)$ on $M$ without $\di \eta(X,Y)=g(X, \phi Y).$

One can also observe that every contact manifold with an almost contact metric structure $(\phi,\xi,\eta,g)$ satisfying $(\nabla_X \phi)X=0,$ or equivalently $(\nabla_X \phi)Y+(\nabla_Y \phi)X=0$, i.e., with a \emph{nearly cosymplectic structure}, satisfy the following condition
\begin{equation} \label{*} 
 \phi \circ \nabla \xi + \nabla \xi \circ \phi =0\tag{$*$}
\end{equation}
and of course does not satisfy the contact metric condition $\di \eta(X,Y)=g(X, \phi Y).$ Here $\nabla$ denotes the Levi-Civita connection of $g$ and $\nabla \xi$ is the bundle endomorphism of $TM$ defined by 
$X\mapsto \nabla _X \xi$. A well-known example of this situation is given by the five-dimensional sphere $S^5.$ This is a consequence of the following theorem (\cite{blair2010riemannian}, Theorem 6.14)
\begin{theorem*}
Let $i:M^{2n+1}\rightarrow \tilde{M}^{2n+2}$ be a hypersurface of a nearly K$\ddot{a}$hler manifold $(\tilde{M}^{2n+2},J,\tilde{g}).$ Then the induced almost contact structure $(\phi,\xi,\eta,g)$ satisfies $(\nabla_X\phi)X=0$ if and only if the second fundamental form $\sigma$ is proportional to $(\eta\otimes \eta)Ji_*\xi.$
\end{theorem*}
If we consider $S^5$ as a totally geodesic hypersurface of $S^6,$ we have that the nearly K\"{a}hler structure $(J,\tilde{g})$ on $S^6,$ defined as in Example 4.5.3 of \cite{blair2010riemannian}, induces an almost contact metric structure $(\phi,\xi,\eta,g)$ on $S^5$ satisfying $(\nabla_X\phi)X=0.$

In the next section we will treat contact manifolds with an almost contact metric structure satisfying condition \eqref{*}. Such manifolds will result of dimension $4n+1$, $n\geqslant 1$. If we suppose that $\phi$ is $\eta$-parallel and the sectional curvature of the horizontal $2$-planes is constant at one point, then we obtain that these manifolds have dimension $5$ (Theorem~\ref{principale}).

It is well known that the contact condition imposes strong restrictions on the Riemannian curvature of an associated metric.
For example Z. Olszak in \cite{Olszak} proves that if an associated metric has constant curvature, then $c=1$ and $g$ must be a Sasakian metric; earlier D.E. Blair in \cite{blair1976} showed that in dimension $\geqslant 5$ there are no flat associated metrics.
We obtain that this is sometimes true also in the case of non associated metrics; for example when $g$ is the metric of a nearly cosymplectic structure, see Theorem~\ref{Th_nearly} in Section~\ref{sec:near_cos_case}.

\section{A class of contact manifolds}

Let $(\phi,\xi,\eta,g)$ be an almost contact metric structure on a contact manifold $(M,\eta)$. We denote by $A$ the vector bundle endomorphism $\nabla \xi :TM\rightarrow TM$. 
Let $B:D\rightarrow D$ be the skew-symmetric part of $A|_D,$ i.e.,
$$
B=\frac{1}{2}(A|_D-A^*)
$$ 
where $A^*$ is the adjoint of $A|_{D}$ with respect to $g|_{D\times D}.$ Then, for all $X,Y \in D,$ we have 
\begin{equation}\label{B}
 \di \eta(X,Y)=-\frac{1}{2} \eta ([X,Y])=-\frac{1}{2}g([X,Y],\xi)=g(BX,Y).
\end{equation}
Even if $\eta$ is a contact form, $\xi$ in general is not the Reeb vector field of $\eta .$

\begin{prop}\label{lemmaB}
Let $(\phi,\xi,\eta, g)$ be an almost contact metric structure on a contact manifold $(M,\eta)$ such that 
$$
\di \eta(\phi X,\phi Y)=-\di \eta(X,Y), \text{ for all } X,Y\in D
$$ 
or equivalently 
$$
B\phi+\phi B=0 \text{ on } D.
$$
Then $\dim M=4n+1, n\geqslant 1$ and $B:D\rightarrow D$ is a bundle automorphism.
\end{prop}

\begin{proof} 
We know that if $(M,\eta)$ is a contact manifold then $\di \eta|_{D\times D}$ is non degenerate. Thus equation \eqref{B} imply that $B$ is an automorphism. 
The fact that $\dim M=4n+1$ is an application of Lemma \ref{dim D}, point \ref{2}.
\end{proof}

\begin{lemma} \label{dim D}
Let $<,>$ be an Hermitian scalar product on a complex vector space $(D,J).$ If $A:D\rightarrow D$ is a nonzero linear operator such that $AJ+JA=0,$ then 
\begin{enumerate}
 \item there exist $Y,Z\in D$ such that $Y,J Y,AY$ are linearly independent, $Z\in \spn \{Y,J Y, AY \}^\perp$and $< Z, J AY>\neq 0$;
 \item if $A$ is non singular and skew-symmetric then $ \dim D \equiv 0 \pmod{4}$. \label{2}
\end{enumerate}
\end{lemma}
\begin{proof}
Let $X_1,..,X_n\in D$ be vectors such that $\{X_1,JX_1,..,X_n,JX_n\}$ is a basis of $D.$ We begin by proving the existence of a vector $Y\in D$ such that $Y,J Y,AY$ are linearly independent. If by contradiction $AY\in \spn \{Y,J Y\}$ for all $Y\in D,$ then  
 \begin{equation*}
  \begin{aligned}
   AX_i \in \spn \{X_i,J X_i\},\\
   AJX_i=-JAX_i \in \spn \{ X_i,J X_i\},
  \end{aligned}
 \end{equation*}
and hence $A$ is represented with respect to our basis by a block-diagonal matrix of the form

     $$ \begin{pmatrix}
  \begin{matrix}
            a_1 & b_1\\
            b_1 & -a_1
           \end{matrix} & \mathbf{0} & \cdots & \mathbf{0} \\
  \mathbf{0} & \begin{matrix}
            a_2 & b_2\\
            b_2 & -a_2
           \end{matrix}  & \cdots & \mathbf{0} \\
  \vdots  & \vdots  & \ddots & \vdots  \\
  \mathbf{0} & \mathbf{0} & \cdots  & \begin{matrix}
            a_n & b_n\\
            b_n & -a_n
           \end{matrix}
 \end{pmatrix}$$
 where $\mathbf{0}=\begin{pmatrix}
                    0 &0\\
                    0 &0
                   \end{pmatrix}$
and $a_i,b_i \in \mathbb{R}$, $i\in \{1,..,n\}.$  
Since
 $$
 A(X_i+X_j) \in \spn \{ X_i+X_j,JX_i+JX_j\},
 $$ 
we have $a_i=a_j$ and $b_i=b_j.$ Thus 
 $$ A \equiv \begin{pmatrix}
  \begin{matrix}
            a_1 & b_1\\
            b_1 & -a_1
           \end{matrix} & \mathbf{0} & \cdots & \mathbf{0} \\
  \mathbf{0} & \begin{matrix}
            a_1 & b_1\\
            b_1 & -a_1
           \end{matrix}  & \cdots & \mathbf{0} \\
  \vdots  & \vdots  & \ddots & \vdots  \\
  \mathbf{0} & \mathbf{0} & \cdots  & \begin{matrix}
            a_1 & b_1\\
            b_1 & -a_1
           \end{matrix}
 \end{pmatrix}.$$ 
 Now we consider $JX_1+X_2$. Since
 \begin{equation*}
   A(JX_1+X_2)\in \spn \{JX_1+X_2,-X_1+JX_2\}
  \end{equation*}
 it follows $a_1=b_1=0.$ This contradicts the hypothesis $A\neq 0$.
 
Let $Y\in D$ be such that $Y,J Y,AY$ are linearly independent. We can observe that 
$$
J AY\notin \spn \{Y,J Y,AY\},
$$ 
so that $J AY=W+Z$, with $W\in \spn \{Y,J Y,AY\}$ and $Z\in \spn \{Y,J Y,AY\}^\perp$, $Z\neq 0$. Thus we found $Z\in D$ orthogonal to $Y,J Y,AY$ such that $<Z,J AY>\neq 0$. 

Now we assume that $A$ is non singular and skew-symmetric. Let $X\in D$ be an eigenvector of the symmetric linear operator $A^2.$ Since $A$ anti-commutes with $J,$ we have that $JX,AX,JAX$ are also eigenvectors of $A^2.$ Moreover the vectors $X,JX,AX,JAX$ are pairwise orthogonal and hence $\dim D \geqslant 4.$ 
 
Assume $\dim D >4.$ By the Spectral Theorem we can choose $Y\in D$ eigenvector of $A^2$ orthogonal to $X, JX,AX,JAX.$ We have that
$$
X,JX,AX,J AX,Y,JY,AY,JAY
$$
are eigenvectors of $A^2,$ pairwise orthogonal and hence $\dim D \geqslant 8.$ Iterating this argument we obtain the assertion.
\end{proof}

\vspace{2 mm}
After these preliminaries we can state our main result that involve contact manifolds with an almost contact metric structure satisfying condition \eqref{*}.

\begin{theorem}\label{principale}
Let $(\phi,\xi,\eta, g)$ be an almost contact metric structure on a contact manifold $(M^{2n+1},\eta)$ such that 
\begin{equation}\label{Ap+pA=0}
A \phi+\phi A=0
\end{equation}
\begin{equation}\label{g(nabla phi)=0}
g((\nabla_X \phi)Y,Z)=0
\end{equation}
for each $X,Y,Z\in D$.

Suppose there exist $p\in M$ and $c\in \mathbb{R}$ such that the sectional curvature  $K_p(\pi)=c,$ for each $2-$plane $\pi$ of $D_p.$ Then $\dim M=5$.
Moreover $A_p$ is an isomorphism if and only if $c \neq 0.$

\end{theorem}

\begin{proof}
For each vector field $Z$ on $M$, we denote by $Z^H$ and $Z^V$ the components of $Z$ in $D$ and in its orthogonal complement $D^{\perp}$ respectively. We say that $Z^H$ is the \emph{horizontal part} of $Z$ and $Z^V$ the \emph{vertical part} of $Z$. 
Let $\nabla$ be the Levi-Civita connection of $g.$ We define a new linear connection 
$$
\tilde{\nabla} :=\nabla + H
$$
on $M$ such that for each $X,Y\in D$
\begin{equation*}
 H(X,\xi)=-AX,  \quad H(X,Y)=g(AX,Y)\xi,\quad H(\xi,X)=\frac{1}{2}BX,\quad H(\xi,\xi)=0.
\end{equation*}
Then for each $X,Y \in D$
$$
(\tilde{\nabla}_X \phi)Y=0,
$$
and hence for each $X,Y,Z\in D$ we have that $\tilde{\nabla}_XY\in D$ and also
\begin{equation}\label{R}
\begin{aligned}
\tilde{R} (X,Y)\phi Z-\phi \tilde{R} (X,Y)Z =& \tilde{\nabla}_X \tilde{\nabla}_Y \phi Z-\tilde{\nabla}_Y \tilde{\nabla}_X \phi Z-\tilde{\nabla}_{[X,Y]}\phi Z\\
                                             &-\phi (\tilde{\nabla}_X \tilde{\nabla}_Y Z-\tilde{\nabla}_Y \tilde{\nabla}_X Z-\tilde{\nabla}_{[X,Y]}Z)\\
                                             =&-\tilde{\nabla}_{[X,Y]}\phi Z+\phi \tilde{\nabla}_{[X,Y]}Z\\ 
                                             =&2g(BX,Y)(\tilde{\nabla}_\xi \phi)Z
\end{aligned}                                             
\end{equation}
where $\tilde{R}$ is the curvature tensor of $\tilde{\nabla}.$ On the other hand, for each $X,Y,Z \in D$ we have

\begin{equation*}
\begin{aligned}
\tilde{R}(X,Y)Z =& R(X,Y)Z-H(X,H(Y,Z))+H(Y,H(X,Z))\\
                 & +H(H(X,Y),Z)-H(H(Y,X),Z)+(\tilde{\nabla}_XH)(Y,Z)\\
                 &-(\tilde{\nabla}_YH)(X,Z)
\end{aligned}                 
\end{equation*}
The horizontal part of $\tilde{R}(X,Y)Z$ is given by
\begin{equation*}
\begin{aligned}
 (\tilde{R}(X,Y)Z)^H =&(R(X,Y)Z)^H+g(AY,Z)AX-g(AX,Z)AY\\
                      & +\frac{1}{2}g(AX,Y)BZ-\frac{1}{2}g(AY,X)BZ\\ 
                     =&(R(X,Y)Z)^H+g(AY,Z)AX-g(AX,Z)AY\\
                      & +g(BX,Y)BZ,
\end{aligned}
\end{equation*}
thus 
\begin{equation*}
\begin{aligned}
 (\tilde{R}(X,Y)\phi Z-\phi (\tilde{R}(X,Y)Z))^H =& (R(X,Y)\phi Z)^H+g(AY,\phi Z)AX\\
                                                  & -g(AX,\phi Z)AY+g(BX,Y)B\phi Z\\
                                                  & -\phi ((R(X,Y)Z)^H+g(AY,Z)AX\\
                                                  & -g(AX,Z)AY+g(BX,Y)BZ).
\end{aligned}                                                  
\end{equation*}
Comparing this last equation with \eqref{R} we have 
\begin{equation}\label{g(B)}
\begin{aligned}
2g(BX,Y)((\tilde{\nabla}_\xi \phi)Z-B\phi Z)^H=&(R(X,Y)\phi Z)^H-\phi (R(X,Y)Z)\\
                                               & +g(AY,\phi Z)AX-g(AX,\phi Z)AY\\
                                               &-g(AY,Z)\phi AX+g(AX,Z) \phi AY.
\end{aligned}                                   
\end{equation}
If $c=0,$ i.e., all the sectional curvatures $K_p(\pi)$ with $\pi \subset D_p$ vanish, then for every $X,Y,Z\in D_p$
\begin{equation}\label{W}
\begin{aligned}
2g(BX,Y)((\tilde{\nabla}_\xi \phi)Z-B\phi Z)^H=& g(AY,\phi Z)AX-g(AX,\phi Z)AY\\
                                               &-g(AY,Z)\phi AX+g(AX,Z) \phi AY.
\end{aligned}                                   
\end{equation}
Consider $Y\in D_p$ such that $AY\neq 0.$ Hence if we take $Z=\phi AY$ we have  
\begin{equation}\label{W'}
\begin{aligned}
g(AY,AY)AX=& -2g(BX,Y)((\tilde{\nabla}_\xi \phi)\phi AY+BAY)^H+g(AX,AY)AY\\
                                               &+g(AX,\phi AY) \phi AY
\end{aligned}                                   
\end{equation}
for every $X\in D_p$ and thus $A:D_p\rightarrow D_p$ has rank $\leqslant 3.$ Then there exists $X\in D_p,X\neq 0$ such that $AX=0.$ Then, by \eqref{W} and \eqref{B} we have that 
$$
\di \eta (X,Y)((\tilde{\nabla}_\xi \phi)Z-B\phi Z)^H=0,
$$ 
for each $Y,Z\in D_p.$ Thus, being $\eta$ a contact form, for each $Z\in D_p$
$$
((\tilde{\nabla}_\xi \phi)Z-B\phi Z)^H=0.
$$
In conclusion, the equation \eqref{W'} becomes
\begin{equation*}
g(AY,AY)AX=g(AX,AY)AY+g(AX,\phi AY) \phi AY,                                 
\end{equation*}
for every $X\in D_p$, yielding $\rank (A)\leqslant 2.$ Now the contact condition implies that $\dim (\ker A)\leqslant n.$ 
Thus $2n\leqslant 2+n,$ namely $n\leqslant 2$ and hence $\dim M\leqslant 5.$
On the other hand, observing that \eqref{Ap+pA=0} also implies that $B$ anti-commute with $\phi$, by Proposition~\ref{lemmaB}, we have that $\dim M\geqslant5.$

Now suppose $c\neq 0.$ Then $A:D_p\rightarrow D_p$ is an isomorphism. Indeed, assume $X\in D_p$ such that $A X=0,$ and $Y\in D_p$ orthogonal to $X,\phi X,BX$
(for example take $Y=\phi BX$). For $X_1,X_2,X_3\in D$ we set 
$$S(X_1,X_2,X_3):= \tilde{R}(X_1,X_2)\phi X_3-\phi(\tilde{R}(X_1,X_2)X_3).$$
Then we have
\begin{equation*}
S(X,Y,X)=2g(BX,Y)(\tilde{\nabla}_\xi \phi)X=0;
\end{equation*}
but on the other hand
\begin{eqnarray*}
(S(X,Y,X))^H &=&(R(X,Y)\phi X)^H+g(AY,\phi X)AX-g(AX,\phi X)AY\\
              &&+g(BX,Y)B\phi X-\phi ((R(X,Y) X)^H+g(AY,X)AX\\
              &&-g(AX,X)AY+g(BX,Y)BX)\\
             &=&cg(X,X) \phi Y,
\end{eqnarray*}
so that $X=0.$

Now, supposing that \eqref{Ap+pA=0} holds, we apply Lemma~\ref{dim D}; fix $Y,Z\in D_p$ such that $Z\in \spn \{Y,\phi Y,AY \}^\perp$ and $g(Z,\phi AY)\neq 0$, then the equation \eqref{g(B)} becomes
\begin{equation*}
 \begin{aligned}
  g(AY,\phi Z)AX=&2g(BX,Y)((\tilde{\nabla}_\xi \phi)Z-B\phi Z)^H+cg(\phi Z,X)Y-cg(Z,X)\phi Y\\
                 & +g(AX,\phi Z)AY-g(AX,Z) \phi AY.
 \end{aligned}
\end{equation*}
This implies that $\rank(A)\leqslant 5$, so that $n\leqslant 2$. As before, we conclude that $\dim M=5$.
\end{proof}

From the above proof, we see that in the case $c=0$ one can obtain the assertion replacing the condition \eqref{Ap+pA=0} with the weaker condition 
$$
\di \eta(\phi X,\phi Y)=-\di \eta (X,Y),
$$
i.e. we have the following

\begin{cor}
 Let $(\phi,\xi,\eta, g)$ be an almost contact metric structure on a contact manifold $(M^{2n+1},\eta)$ such that 
\begin{equation*}
\di \eta(\phi X,\phi Y)=-\di \eta (X,Y),
\end{equation*}
\begin{equation*}
g((\nabla_X \phi)Y,Z)=0,
\end{equation*}
for each $X,Y,Z\in D$.
We suppose there exists $p\in M$ such that the sectional curvature  $K_p(\pi)=0,$ for each $2-$plane $\pi$ of $D_p.$ Then $\dim M=5$.
\end{cor}

\vspace{2 mm}

Almost contact metric manifolds are classified by Chinea and Gonzalez in  \cite{chinea1990}. The authors define twelve classes of manifolds $C_1,\dots ,C_{12}$.
All manifolds in the classes $C_i$ for $i \in {\{5,6,..,12\}}$ satisfy the condition \eqref{g(nabla phi)=0}, and all manifolds in $C_9$ or $C_{10}$ satisfy \eqref{g(nabla phi)=0} and \eqref{Ap+pA=0}. Thus we have the following

\begin{theorem}
Every contact manifold $(M,\eta)$ carrying an almost contact metric structure $(\phi, \xi,\eta,g)$ of class $C_9\oplus C_{10}$ has dimension $4n+1$, with $n\geqslant 1$.

If there exists $p\in M$ and $c\in \mathbb{R}$ such that the sectional curvature $K_p(\pi)=c,$ for each $2-$plane $\pi$ of $D_p$, then $\dim M=5$. 
\end{theorem}

\section{Nearly cosymplectic case}\label{sec:near_cos_case}

In this section we will show that there does not exist a flat nearly cosymplectic manifold $(M, \phi,\xi,\eta, g)$ with $\eta$ a contact form.

\begin{lemma} \label{lemmaNearly}
 Let $(M,\phi,\xi,\eta,g)$ be a nearly cosymplectic manifold. Then 
\begin{enumerate} [(a)]
  \item  $\di\eta(X,Y)=g(AX,Y)$ for all $X,Y\in TM$, \label{d eta=g}
  \item $\di\eta(X,Y)=-\di\eta(\phi X,\phi Y)$ for all $X,Y\in TM$,
  \item $\xi$ is the Reeb vector field of $(M^{2n+1},\eta )$.
\end{enumerate}
If moreover $\eta$ is a contact form, then  
 \begin{enumerate} [(a),resume]
  \item for all $p\in M^{2n+1}$ $A_p$ is an isomorphism that anti-commutes with $\phi$, \label{A_iso}
  \item $g((\nabla_X\phi)Y,Z)=0$, for all $X,Y,Z\in D$,\label{ggg}
  \item $\dim M= 4n+1$.
\end{enumerate}
\end{lemma}

\begin{proof}
Let $\nabla$ be the Levi-Civita connection of $g.$ Since $\xi$ is Killing, we have
 \begin{eqnarray*}
    2g(AX,Y) &=& 2g(\nabla_X \xi,Y)\\
             &=& X(g(\xi,Y))+ \xi(g(Y,X)) -Y(g(X,\xi))\\
               && +g([X,\xi],Y)-g([\xi,Y],X)+g([Y,X],\xi)\\
             &=&  X(g(\xi,Y))-Y(g(X,\xi))+g([Y,X],\xi)\\
             &=& X(\eta (Y))-Y(\eta (X))-\eta ([X,Y])\\
             &=& 2\di\eta (X,Y)
 \end{eqnarray*}
 for all $X,Y\in TM$. By Lemma 3.1 of \cite{endo2005curvature} we have that 
 \begin{equation*}
 A\phi +\phi A=0.
 \end{equation*}
Then 
$$
\di\eta (\phi X,\phi Y)=g(A\phi X, \phi Y)=-g(AX,Y)=-\di\eta (X,Y),
$$
from which it follows that $$\di\eta (X,\xi)=-\di\eta (\phi X,\phi \xi)=0.$$ 
As a consequence of \ref{d eta=g} we have that $A_p$ is an isomorphism. Finally \ref{ggg} follows from \ref{A_iso} and the equation
\begin{equation*}
 g((\nabla_X\phi)Y,AZ)=\eta (Y)g(A^2X,\phi Z)-\eta (X)g(A^2Y,\phi Z)
\end{equation*}
due to H. Endo \cite{endo2005curvature}.
\end{proof}

Hence, as a consequence of Theorem~\ref{principale}, we can state
\begin{theorem} \label{Th_nearly}
 Let $(M^{2n+1},\eta)$ be a contact manifold endowed with a nearly cosymplectic structure $(\phi,\xi,\eta,g).$ 
 
 Suppose there exist $p\in M$ and $c\in \mathbb{R}$ such that for each $2-$plane $\pi$ of $D_p,$ $K_p(\pi)=c.$  Then $c\neq 0$ and $\dim M=5$.
\end{theorem}

\begin{rem}
\emph{H. Endo in \cite{endo2005curvature} determines the curvature tensor of a nearly cosymplectic manifold $(M,\phi,\xi,\eta,g)$ with pointwise constant $\phi$-sectional curvature $c$
\begin{equation}\label{endo}
\begin{aligned}
 4g(R(W,X)Y,Z)=&g((\nabla_W \phi)Z,(\nabla_X \phi)Y)-g((\nabla_W \phi)Y,(\nabla_X \phi)Z)\\
               &-2g((\nabla_W \phi)X,(\nabla_Y \phi)Z)+g(\nabla_W \xi ,Z)g(\nabla_X \xi,Y)\\
               &-g(\nabla_W \xi,Y)g(\nabla_X \xi,Z)-2g(\nabla_W \xi,X)g(\nabla_Y \xi,Z)\\
               &-\eta(W)\eta(Y)g(\nabla_X \xi,\nabla_Z \xi)+\eta(W)\eta(Z)g(\nabla_X \xi,\nabla_Y \xi)\\
               &+\eta(X)\eta(Y)g(\nabla_W \xi,\nabla_Z \xi)-\eta(X)\eta(Z)g(\nabla_W \xi,\nabla_Y \xi)\\
               &+c\{g(X,Y)g(Z,W)-g(Z,X)g(Y,W)\\
               &+\eta(Z)\eta(X)g(Y,W)-\eta(Y)\eta(X)g(Z,W)\\
               &+\eta(Y)\eta(W)g(Z,X)-\eta(Z)\eta(W)g(Y,X)\\
               &+g(\phi Y,X)g(\phi Z,W)-g(\phi Z,X)g(\phi Y,W)\\
               &-2g(\phi Z,Y)g(\phi X,W)\}.
\end{aligned}
\end{equation}
One can obtain the conclusion of Theorem~\ref{Th_nearly} also using this formula together with Lemma~\ref{lemmaNearly}. If there exists a point $p\in M$ such that the sectional curvature of all the $2$-planes of $D_p$ is constant, then for all $X,Y,W\in D$ we have
\begin{equation*}
 \begin{aligned}
R(W,X)Y&=c(g(Y,X)W-g(Y,W)X),\\
g((\nabla_W \phi)Z,(\nabla_X \phi)Y)&=g(\phi Y,AX)g(\phi Z,AW),
 \end{aligned}
\end{equation*}
and \eqref{endo} becomes
\begin{equation*}
 \begin{aligned}
 3c(g(Y,X)W-g(Y,W)X)=&-g(\phi Y,AX)\phi AW+g(\phi Y,AW)\phi AX\\
               &+2g(\phi X, AW)\phi AY+g(AX,Y)AW\\
               &-g(AW,Y)AX-2g(AW,X)AY\\
               &+c\{-g(X,\phi Y)\phi W+g(\phi Y,W)\phi X\\
               &+2g(\phi X,W)\phi Y \}.
 \end{aligned}
\end{equation*}
If in particular $Y=AW$, then 
\begin{equation*}
\begin{aligned}
  3cg(X,AW)W=&\{-g(\phi AW,AX)+2cg(\phi X,W)\}\phi AW+2g(\phi X,AW)\phi A^2 W\\
             &+g(AX,AW)AW-g(AW,AW)AX-2g(AW,X)A^2W\\
             &-cg(\phi AW,X)\phi W,
\end{aligned}
\end{equation*}
and hence $ \rank(A) \leqslant 6.$ By Lemma~\ref{lemmaNearly} it follows that $\dim M=5$.}

\end{rem}

\section*{Acknowledgments}
The author thanks Antonio Lotta for useful discussions and suggestions.

\vspace{3 mm}
\noindent 
Eugenia Loiudice

\vspace{1 mm}
\noindent Dipartimento di Matematica, Università di Bari ``Aldo Moro'', 

\noindent Via Orabona 4, 70125 Bari, Italy

\noindent \emph{e-mail}: eugenia.loiudice@uniba.it


\begin{thebibliography}{0}

\bibitem{blair1976} D.~E. Blair,
\emph{On the non-existence of flat contact metric structures},
Tohoku Math. J. (2), 28(3) (1976), 373--379.

\bibitem{blair2010riemannian} D.~E. Blair,
\emph{Riemannian geometry of contact and symplectic manifolds,
  volume 203 of Progress in Mathematics},
 Birkh\"auser Boston, Inc., Boston, MA, second edition, 2010.

\bibitem{chern} S.~S. Chern,
\emph{Pseudo-groupes continus infinis},
Colloques Internationaux du C. N. R. S., Strasbourg (1953), 119--136.

\bibitem{chinea1990} D.~Chinea and C.~Gonzalez,
\emph{A classification of almost contact metric manifolds},
Ann. Mat. Pura Appl. (4), 156(1) (1990), 15.

\bibitem{endo2005curvature} H.~Endo,
\emph{On the curvature tensor of nearly cosymplectic manifolds of constant
  $\varphi$-sectional curvature},
An. Stiint. Univ. Al. I. Cuza Iasi. Mat. (N.S.), 51 (2005), 439--454.

\bibitem{Olszak} Z.~Olszak,
\emph{On contact metric manifolds},
Tohoku Math. J. (2), 31(2)(1979), 247--253.

\end{thebibliography}
\end{document}